\documentclass[11pt]{amsart}
\usepackage{amssymb,amsthm,amsmath,amstext}
\usepackage{mathrsfs}  
\usepackage{mathtools} 
\usepackage[usenames,dvipsnames,svgnames,table]{xcolor}
\usepackage{graphicx}
\usepackage[all]{xy}

\usepackage[left=1.5in,top=1.5in,right=1.5in,bottom=1.5in]{geometry}

\numberwithin{equation}{section}

\theoremstyle{plain}
\newtheorem{prop}{Proposition}
\newtheorem{theo}[prop]{Theorem}

\newtheorem{lemm}[prop]{Lemma}

\newtheorem{theoremintro}{Theorem}

\theoremstyle{definition}
\newtheorem{defi}[prop]{Definition}

\newtheorem{rema}[prop]{Remark}

\DeclareMathOperator{\Pic}{\mathrm{Pic}}
\DeclareMathOperator{\Br}{\mathrm{Br}}
\DeclareMathOperator{\Brur}{\mathrm{Br}_{\ur}}

\newcommand{\CH}{\mathrm{CH}}
\newcommand{\Spec}{\mathrm{Spec}}
\newcommand{\im}{\mathrm{im}}
\newcommand{\cor}{\mathrm{cor}}

\newcommand{\isom}{\cong}
\newcommand{\simto}{\mathrel{\vcenter{\offinterlineskip\hbox{\hskip.4ex$\sim$}\vskip-0.7ex\hbox{$\xrightarrow{~~\;}$}}}}
\newcommand{\isomto}{\simto}

\newcommand{\Z}{\mathbb Z}
\newcommand{\A}{\mathbb A}
\newcommand{\C}{\mathbb C}
\renewcommand{\P}{\mathbb P}
\newcommand{\Q}{\mathbb Q}
\newcommand{\Gm}{\mathbb{G}_{\mathrm{m}}}
\newcommand{\SL}{\mathrm{SL}}

\newcommand{\sep}{^{{s}}}

\newcommand{\sing}{^{\mathrm{sing}}}

\newcommand{\dual}{^{\vee}}
\newcommand{\tensor}{\otimes}

\newcommand{\mapto}[1]{\xrightarrow{#1}}

\newcommand{\et}{\mathrm{\acute{e}t}}
\newcommand{\Het}{H_{\et}}
\newcommand{\ur}{\mathrm{nr}}
\newcommand{\Hur}{H_{\ur}}
\newcommand{\Hnr}{H_{\ur}}

\newcommand{\HH}{\mathcal{H}}
\newcommand{\KK}{\mathcal{K}}
\newcommand{\OO}{\mathscr{O}}

\newcommand{\linedef}[1]{\textsl{#1}}

\usepackage[backref=page]{hyperref}
\hypersetup{pdftitle={Stable rationality of quadric and cubic surface bundle fourfolds}}
\hypersetup{pdfauthor={Asher Auel, Christian Boehning, Alena Pirutka}}
\hypersetup{colorlinks=true,linkcolor=blue,anchorcolor=blue,citecolor=blue}

\begin{document}

\title[Stable rationality of quadric and cubic surface bundles]{Stable
rationality of quadric and cubic surface bundle fourfolds}

\author[Auel]{Asher Auel}
\address{Asher Auel, Department of Mathematics\\
Yale University\\
New Haven, Connecticut 06511, United States}
\email{asher.auel@yale.edu}

\author[B\"ohning]{Christian B\"ohning}
\address{Christian B\"ohning, Mathematics Institute, University of Warwick\\
Coventry CV4 7AL, England}
\email{C.Boehning@warwick.ac.uk}

\author[Pirutka]{Alena Pirutka}
\address{Alena Pirutka, Courant Institute of Mathematical Sciences,
New York University\\ 
New York, New York 10012, United States\newline
National Research University Higher School of Economics, Russian Federation}
\email{pirutka@cims.nyu.edu}

\maketitle

\begin{abstract}
We study the stable rationality problem for quadric and cubic surface
bundles over surfaces from the point of view of the degeneration
method for the Chow group of $0$-cycles.  Our main result is that a
very general hypersurface $X$ of bidegree $(2,3)$ in $\P^2 \times
\P^3$ is not stably rational.  Via projections onto the two factors,
$X\to \P^2$ is a cubic surface bundle and $X\to \P^3$ is a conic
bundle, and we analyze the stable rationality problem from both these
points of view.  Also, we introduce, for any $n\geq 4$, new quadric
surface bundle fourfolds $X_n \to \P^2$ with discriminant curve $D_n
\subset \P^2$ of degree $2n$, such that $X_n$ has nontrivial
unramified Brauer group and admits a universally $\CH_0$-trivial
resolution.
\end{abstract}

\vspace{.6cm}

\section{Introduction}

An integral variety $X$ over a field $k$ is {\it stably rational} if
$X\times \P^m$ is rational, for some $m$.  In recent years,
failure of stable rationality has been established for many classes of
smooth rationally connected projective complex varieties, see, for
instance \cite{Aokada, auel-co, beau-6, bohboth, CTP-cyclic, CTP,
HKT-conic, HPT, HPT-double, HPT-quad, HT-fano,
krylov_okada:del_Pezzo_fibrations_low_degree, okada,
okada:weighted_hypersurfaces, okada:index_one_Fano, Sch2, Sch1,
totaro-JAMS, voisin}. These results were obtained by the
specialization method, introduced by C.\ Voisin \cite{voisin} and
developed in \cite{CTP}.  In many applications, one uses this method
in the following form. For simplicity, assume that $k$ is an
uncountable algebraically closed field and consider a quasi-projective
integral scheme $B$ over $k$ and a generically smooth projective
morphism $\mathcal X \to B$ with positive dimensional fibers.  In
order to prove that a very general fiber $\mathcal X_b$ of this family
is not stably rational it suffices to exhibit a single integral fiber
$Y=\mathcal X_{b_0}$, typically singular, such that:
\begin{itemize}
\item[(R)] $Y$ admits a universally $\CH_0$-trivial resolution $\tilde
Y \to Y$ of singularities,

\item[(O)] $\tilde Y$ is not universally $\CH_0$-trivial, e.g., the
function field $k(Y)$ admits a nontrivial \'{e}tale unramified
invariant such as $\Hnr^2(k(Y)/k, \Q/\Z(1))$.
\end{itemize}
We call such $Y$ a ``reference variety,'' see Definition
\ref{def:ref_var}.

Recall \cite{ACTP:unramified_cohomology_cubic_fourfold}, \cite{CTP}
that a proper variety $X$ over $k$ is \linedef{universally
$\CH_0$-trivial} if for every field extension $k'/k$, the degree
homomorphism on the Chow group of 0-cycles $\CH_0(X_{k'})\rightarrow
\Z$ is an isomorphism.  A proper morphism $f: \tilde X\rightarrow X$
of $k$-varieties is \linedef{universally $\CH_0$-trivial} if for every
field extension $k'/k$, the push-forward homomorphism $f_* :
\CH_0(\tilde {X}_{k'})\rightarrow \CH_0(X_{k'})$ is an isomorphism.
Then a \linedef{universally $\CH_0$-trivial resolution} of $X$ is a
proper birational universally $\CH_0$-trivial morphism $f:\tilde X\to
X$ with $\tilde X$ smooth.

In \cite{HPT}, the specialization method was applied to show that a
very general hypersurface of bidegree $(2,2)$ in $\P^2\times \P^3$ is
not stably rational over $\C$, utilizing the following reference variety:
\begin{equation}\label{refHPT}
Y \;\;:\;\; yzs^2+xzt^2+xyu^2+(x^2+y^2+z^2-2xy-2xz-2yz)v^2=0.
\end{equation}
Such hypersurfaces have the structure of a quadric surface bundle over
$\P^2$, by projection to the first factor, which inform the shape of
the equation for the reference variety $Y$ as in
\cite{pirutka-survol}. In \cite{HPT}, it was also shown that the locus,
in the Hilbert scheme of all hypersurfaces of bidegree $(2,2)$ in
$\P^2\times \P^3$, where the quadric bundle admits a rational section,
is dense (for the complex topology).  This provided a first example
showing that rationality is not a deformation invariant for smooth
projective complex varieties.  A detailed analysis of this same
reference variety, from the point of view of the conic bundle
structure obtained by projection onto the second factor, is made in
\cite{auel-co}.

Recently, Schreieder~\cite{Sch2} developed a refinement of the
specialization method, relaxing the condition that the reference
variety admits a universally $\CH_0$-trivial resolution.  This helped
establish the failure of stable rationality for many families of
quadric bundles \cite{Sch2}, and in particular a large class of
quadric surface bundles over $\P^2$ of graded free type \cite{Sch1}.



In a different direction, recent results of Ahmadinezhad and
Okada~\cite{Aokada} imply the failure of stable rationality for many
families of conic bundles over projective space, including a very
general hypersurface of bidegree $(2,d)$ in $\P^2 \times \P^3$ for $d
\geq 4$.  In this work, the specialization method is also used, where
the reference varieties constructed have global differential forms in
characteristic $p$, following the method of Koll\'ar developed by
Totaro~\cite{totaro-JAMS}.

The goal of this note is threefold.  First, we complete the stable
rationality analysis for hypersurfaces of bidegree $(2,d)$ in $\P^2
\times \P^3$.

\begin{theoremintro} 
\label{thm:main_intro} 
The very general hypersurface of bidegree $(2,3)$ in $\P^2 \times
\P^3$ over $\C$ is not stably rational.
\end{theoremintro}

It is easy to produce loci of bidegree $(2,3)$ hypersurfaces in $\P^2
\times \P^3$ that are rational, giving another example of a family of
smooth rationally connected fourfolds with some fibers rational but
most not stably rational, see Remark~\ref{rem:somerational}.

We provide two different proofs.  Our first proof, in \S\ref{sec:23},
uses a method, going back to Totaro~\cite{totaro-JAMS} for
hypersurfaces in projective space, for reducing the case of
hypersurfaces of bidegree $(2,d+1)$ in $\P^n \times \P^m$ to those of
bidegree $(2,d)$, by constructing reference varieties that are
reducible, see also \cite{bohboth}.  This method only works in general
when $m=2$, but with some additional geometric construction relying
heavily on the analysis in \cite{auel-co}, we are able to handle the
case $m=3$.  We remark that hypersurfaces of bidegree $(2,3)$ in $\P^2
\times \P^3$ have the structure of a conic bundle over $\P^3$ by
projection onto the second factor, and a cubic surface bundle over
$\P^2$ by projection onto the first factor.  For our second proof, in
\S\ref{sec:cubic_example}, we construct a new reference hypersurface
of bidegree $(2,3)$ in $\P^2 \times \P^3$ such that the associated
cubic surface bundle is generically smooth as opposed to in our first
construction.

Our second goal, in \S\ref{sec:cubic}, is to more generally study
cubic surface bundles over a rational surface, with an aim toward
investigating properties of the discriminant curve (which, in the case
of a cubic surface bundle, arises from the work of
Salmon~\cite{salmon:discriminant},
Clebsch~\cite{clebsch:quaternary_cubic},
\cite{clebsch:quaternary_cubic_II}, and clarified by
Edge~\cite{edge:discriminant_cubic_surface}) and the ramification
profile, in the spirit of the quadric surface bundle case in
\cite{pirutka-survol}.
This analysis provides 
an example of a smooth cubic surface $X$ over $K=\C(x,y)$ such that
the cokernel of the natural map $\Br(K)\to \Br(X)$ contains a
nontrivial 2-torsion Brauer class that is unramified on any fourfold
model of $X$, see Remark~\ref{rem:nonconstant}.

Finally, our third goal, in \S\ref{sec:poly}, is to provide a new
family of reference fourfolds $X_n \to \P^2$, whose generic fiber is a
diagonal quadric, and whose discriminant curve of degree $2n \geq 8$
consists of an $(n-1)$-gon of double lines with an inscribed conic.
These are a generalization of \eqref{refHPT} (which occurs as $X_4$).
The varieties in this family have nontrivial unramified Brauer group
by an application of the general formula in \cite{pirutka-survol}.
Like for the family of reference varieties constructed by
Schreieder~\cite{Sch1}, the morphism $X_n \to \P^2$ need not be flat,
however, in contrast, each reference variety in our family admits a
$\CH_0$-universally trivial resolution.  As a corollary, we obtain
particular cases of \cite[Theorem 1]{Sch1} using the specialization
method.





\medskip

\paragraph{\bf Acknowledgements} 
The authors would like to thank the organizers of the Simons
Foundation conference on Birational Geometry held in New York City,
August 21--25, 2017, as well as the Laboratory of Mirror Symmetry and
Automorphic forms, HSE, Moscow, where some of this work was
accomplished. The first author was partially supported by NSA Young
Investigator Grant H98230-16-1-0321. The third author was partially
supported by NSF grant DMS-1601680 and by the Laboratory of Mirror
Symmetry NRU HSE, RF Government grant, ag.\ no.\ 14.641.31.0001. The
authors would like to thank Jean-Louis Colliot-Th\'el\`ene, Bjorn
Poonen, Stefan Schreieder, and Yuri Tschinkel for helpful comments.

\section{Hypersurfaces of bidegree $(2,3)$ in $\P^2 \times \P^3$}
\label{sec:23}

In this section we study hypersurfaces $X$ of bidegree $(2,3)$ in
$\P^2 \times \P^3$ over $\C$.  Via projection onto the two factors, $X
\to \P^2$ is a cubic surface bundle and $X \to \P^3$ is a conic
bundle.  We point out that of the recent results \cite{Aokada},
\cite{bohboth}, \cite{krylov_okada:del_Pezzo_fibrations_low_degree},
\cite{Sch1}, \cite{Sch2} on stable rationality relevant to this case
(e.g., for conic bundles), none actually cover the case of bidegree
$(2,3)$ in $\P^2 \times \P^3$.

\begin{defi}
\label{def:ref_var}
We call a proper variety $Y$ over an algebraically closed field $k$ a
\linedef{reference variety} if for any discrete valuation ring $A$
with fraction field $K$ and residue field $k$ and for any flat and
proper scheme $\mathcal X$ over $A$ with smooth connected generic
fiber $X$ and special fiber $Y$, the $K\sep$-scheme $X_{K\sep}$ is not
universally $\CH_0$-trivial.
\end{defi}

For example, if $Y$ is integral and satisfies conditions (R) and (O)
from the Introduction, then $Y$ is a reference variety by \cite{CTP}.
More generally, examples of reducible reference varieties were
utilized in \cite{totaro-JAMS}.  We recall sufficient conditions
for a reducible variety to be a reference variety.

\begin{prop}
\label{prop:reducible}
Let $Y$ be a proper scheme with two irreducible components $Y_1$ and
$Y_2$ over an algebraically closed field $k$ such that $Y_1 \cap Y_2$
is irreducible.  If $Y_1 \cap Y_2$ is universally
$\CH_0$-trivial and $Y_1$ admits a universally
$\CH_0$-trivial resolution ${f:Y_1' \to Y_1}$ such that $Y_1'$ is not
universally $\CH_0$-trivial, then $Y$ is a reference variety.
\end{prop}
\begin{proof}
Let $U_i=Y_i\smallsetminus (Y_1\cap Y_2)$.  Since $Y_1 \cap Y_2$ is
universally $\CH_0$-trivial and $Y_1$ is not universally
$\CH_0$-trivial, there exists a field extension $l/k$ such that in the
localization exact sequence over~$l$
$$
\CH_0((Y_1\cap Y_2)_l)\to \CH_0(Y_{1,l})\stackrel{\iota}{\to}
\CH_0(U_{1,l})\to 0$$
the first map is not surjective. Hence, there is a non-trivial 0-cycle
$\xi\in \CH_0(U_{1,l})$.  Since $k$ is algebraically closed, $U_1(k)
\neq \varnothing$, thus we can also assume $\xi$ has degree $0$.

We can furthermore assume that $\xi$ is supported on the smooth locus of
$U_1$. Indeed, let $V\subset U_1$ be a smooth open subscheme such
that $f$ induces an isomorphism $f^{-1}(V)\isomto V$, and
consider the following commutative diagram
$$
\xymatrix{
\CH_0(Y'_{1,l})\ar[r]\ar[d]^{f_*}& \CH_0(f^{-1}(U_{1,l}))\ar[d]^{f_*}&\\
\CH_0(Y_{1,l})\ar[r]& \CH_0(U_{1,l}).&\\
}
$$
The left vertical map is an isomorphism and the horizontal maps are
surjective. We deduce that $\xi$ comes from a non-trivial element
$\xi'\in \CH_0(f^{-1}(U_{1,l}))$. Since $f^{-1}(U_{1,l})$ is smooth,
by a moving lemma for 0-cycles \cite[p.~599]{colliot-thelene:finitude},
we can assume that $\xi'$ is supported on the open subscheme
$f^{-1}(V_l)$ and, hence, that $\xi$ is supported on~$V_l$.

By construction, the image of $\xi$ in $\CH_0(Y_l)$ is nonzero.
In fact, this follows from the commutative diagram
$$
\xymatrix{
\CH_0((Y_1\cap Y_2)_l)\ar[r]\ar[d]& \CH_0(Y_{1,l})\ar[r]\ar[d]& \CH_0(U_{1,l})\ar[d]\ar[r]& 0\\
\CH_0((Y_1\cap Y_2)_l)\ar[r]& \CH_0(Y_{l})\ar[r]& \CH_0(U_{1,l})\oplus \CH_0(U_{2,l})\ar[r]& 0
}
$$
where the horizontal sequences are the localization exact sequences,
the leftmost vertical map is an isomorphism, the middle vertical map
is induced by the inclusion $Y_1\to Y$, and the rightmost vertical map
is injective.

By the argument above, we can also assume that $\xi$   is supported on the smooth locus of $Y_l$. Then \cite[Lemma~2.4]{totaro-JAMS} shows that $Y$ is a reference variety.
\end{proof}

We remark that there is a version of this Proposition with $Y$ having
multiple irreducible components whose intersections have multiple
components. In what follows, though, we only use the case $Y_1\cap
Y_2$ irreducible.

We briefly mention two conditions ensuring that a proper variety $Y$
is universally $\CH_0$-trivial.  First, $Y$ is universally
$\CH_0$-trivial if there exists a proper surjective morphism $Y' \to
Y$ with $Y'$ universally $\CH_0$-trivial and such that for any field
extension $k'/k$ and any scheme theoretic point $y \in Y_{k'}$, the
fiber $Y'_{y}$ has a 0-cycle of degree 1.  Second, $Y$ is universally
$\CH_0$-trivial if it admits a universally $\CH_0$-trivial resolution
by a universally $\CH_0$-trivial variety.

For hypersurfaces, a lemma implied by Proposition~\ref{prop:reducible}
was utilized by Totaro to arrive at an inductive procedure for
investigating the stable rationality of a smooth hypersurface $W$ of
degree $2n+1$ in projective space by degenerating $W$ to the union of
a smooth hypersurface of degree $2n$ and a hyperplane.  In
\cite[\S4]{bohboth}, a similar inductive procedure is used,
degenerating a hypersurface of bidegree $(2,n+1)$ in $\P^2 \times
\P^2$ to the union of hypersurfaces of bidegree $(2,n)$ and $(0,1)$.
In this later case, the intersection of hypersurfaces of bidegree
$(2,n)$ and $(0,1)$ in $\P^2 \times \P^2$ can be chosen to be smooth
and has the structure of a conic bundle over $\P^1$, hence is rational
and thus universally $\CH_0$-trivial.  However, attempting this for a
hypersurface of bidegree $(2,n+1)$ in $\P^2 \times \P^3$ is more
subtle.  One can still degenerate to the union of hypersurfaces of
bidegree $(2,n)$ and $(0,1)$, but now the intersection of these
components has the structure of a conic bundle over $\P^2$, which can
certainly fail to be universally $\CH_0$-trivial, cf.\
\cite{HKT-conic}, \cite{bohboth}.  We shall overcome this problem for
hypersurfaces of bidegree $(2,3)$ in $\P^2 \times \P^3$ by using the
explicit geometry of the conic bundle structure on the reference
variety \eqref{refHPT} studied in \cite{auel-co}.

In our construction, we start with the reference variety
\eqref{refHPT}, a singular hypersurface $Y_1$ of bidegree $(2,2)$ in
$\P^2 \times \P^3$.  We recall from \cite[\S3]{auel-co} that
projection to $\P^3$ gives the structure of a conic bundle $Y_1 \to
\P^3$ defined by the $3\times 3$ matrix 
\begin{gather}\label{eq:HPT_conic}
\begin{pmatrix}
v^2 & u^2-v^2 & t^2-v^2 \\
u^2-v^2 & v^2 & s^2-v^2 \\
t^2-v^2 & s^2-v^2 & v^2 
\end{pmatrix}
\end{gather}
of homogeneous quadratic forms on $\P^3$.  The discriminant of this
conic bundle is the sextic surface $D \subset \P^3$ defined by
$$
4v^6 - 4(s^2+t^2+u^2) v^4 + (s^2+t^2+u^2)^2 v^2 - 2 s^2 t^2 u^2 = 0
$$
which has two irreducible cubic surface components $D_\pm$,
defined by
$$
2 v^3 - v (s^2+t^2+u^2) \pm \sqrt{2} stu = 0.
$$
Each component $D_\pm$ is a tetrahedral Goursat surface
\cite{goursat}, hence up to projective equivalence, is isomorphic to
the Cayley nodal cubic surface.  The intersection ${D_+ \cap D_-}$ is an
arrangement of a triangle of lines and three conics, see
\cite[Fig.~1]{auel-co}.  A Cayley cubic surface contains 4 ordinary
double points and 9 lines: 6 of the lines form the edges of a
tetrahedron whose vertices are the 4 singular points, while the
remaining 3 lines form a triangle not meeting the singular points.  In
our case, $D_+ \cap D_-$ does not contain the ordinary double points
of either component, hence contains this later triangle of lines, which is
thus common to both Cayley cubic surface $D_+$ and~$D_-$.

Our aim is then to choose an appropriate hyperplane $H \subset \P^3$
such that the configuration $Y_1 \cup (\P^2 \times H)$, thought of as a
union of hypersurfaces of bidegree $(2,2)$ and $(0,1)$, is a reference
variety for hypersurfaces of bidegree $(2,3)$ in $\P^2 \times \P^3$.
Since $Y_1$ admits a universally $\CH_0$-trivial resolution that is not
universally $\CH_0$-trivial by \cite{HPT} or \cite{auel-co}, then by
Proposition~\ref{prop:reducible} we only need to verify, for our
choice of $H \subset \P^3$, that $Y_1 \cap (\P^2 \times H)$ admits a
universally $\CH_0$-trivial resolution by a smooth variety that is
universally $\CH_0$-trivial.

We remark that $Y_1 \cap (\P^2 \times H) = Y_1|_H \to H$ is simply the
restriction of the conic bundle $Y_1 \to \P^3$ to $\P^2 = H \subset
\P^3$, whose discriminant divisor is now $D \cap H \subset H$.  In
particular, $Y_1|_H \subset \P^2 \times H = \P^2 \times \P^2$ is a
hypersurface of bidegree $(2,2)$.  By \cite{bohboth}, the very general
such hypersurface is not universally $\CH_0$-trivial, so we would
expect that for a very general choice of $H \subset \P^3$, the variety
$Y_1|_H$ would \emph{not} be universally $\CH_0$-trivial.  Indeed, in
our case, we can verify this explicitly.  Let $\alpha\in
H^2(\C(\P^3),\Z/2)$ be the Brauer class corresponding to the generic
fiber of the conic bundle $Y_1 \to \P^3$ and let $\gamma_\pm \in
H^1(\C(D_\pm),\Z/2)$ be its residue class on the component $D_\pm$ of
the discriminant. By the analysis in \cite[\S3]{auel-co}, we know that
$\gamma_\pm$ is \'etale away from the 4 singular points of $D_\pm$.
The residues of the conic bundle $Y_1|_H \to H$ on the components of
its discriminant $D \cap H = (D_+ \cap H) \cup (D_- \cap H)$ are simply
the restrictions of $\gamma_\pm$.  A general hypersurface $H \subset
\P^3$ will cut $D$ in the union of two smooth elliptic curves $E_+
\cup E_-$ meeting transversally and the restriction of $\gamma_\pm$ to
$E_\pm$ are \'etale and would be nontrivial by a suitable version of
the Lefschetz hyperplane theorem.  Hence, by a formula due to
Colliot-Th\'el\`ene (cf.~\cite{pirutka-survol}), the unramified Brauer
group of $Y_1|_H$ would have nontrivial 2-torsion, and since it only
has isolated nodes, it admits a universally $\CH_0$-trivial
resolution by a smooth projective variety that would \emph{not} be universally $\CH_0$-trivial.  In
conclusion, we need to choose the hyperplane $H \subset \P^3$ in a
special way.

We choose the plane $H \subset \P^3$ to be the unique plane spanned by
the triangle of lines contained in $D_+ \cap D_-$.  Then from the
explicit representation \eqref{eq:HPT_conic} of $Y_1$ as a conic
bundle, $H = \{v=0\}$ and we have that $Y_0 = Y_1|_H \to H$ is the
hypersurface of bidegree $(2,2)$ in $\P^2 \times \P^2$ defined by
\begin{equation}\label{eq:HPT_section}
xy u^2 + xz t^2 + yz s^2 = 0
\end{equation}
where $(x:y:z)$ and $(u:t:s)$ are sets of homogeneous coordinates on
$\P^2$.  Given the above discussion, Theorem~\ref{thm:main_intro} will
follow from the following.

\begin{prop}
\label{prop:resolution_HPT_section}
The hypersurface $Y_0$ of bidegree $(2,2)$ in $\P^2 \times \P^2$
defined by \eqref{eq:HPT_section} is a singular projective rational
variety admitting a universally $\CH_0$-trivial resolution.
\end{prop}

We remark that since $Y_0$ is rational, any smooth proper variety
birational to $Y_0$ will be $\CH_0$-universally trivial,
cf.~\cite[\S1.2]{ACTP:unramified_cohomology_cubic_fourfold}.  Hence to
apply Proposition~\ref{prop:reducible}, we only need to verify that
$Y_0$ admits a universally $\CH_0$-trivial resolution.  

\begin{proof}[Proof of Proposition~\ref{prop:resolution_HPT_section}]
The singular locus $Y_0\sing$ of $Y_0$ is the union of three curves:
\begin{align}
\begin{split}
C_x & {}\;\; : \;\; x=s=0,\; u^2y + t^2z=0\\
C_y & {}\;\; : \;\; y=t=0,\; u^2y + s^2x=0\\
C_z & {}\;\; : \;\; z=u=0,\; t^2x+s^2y=0.
\end{split}
\end{align}
The intersection of $Y_0\sing$ with the chart $z\neq 0$ in $\P^2
\times \P^2$ is contained in the affine chart $\A^4 \subset \P^2
\times \P^2$ with coordinates $(x,y,s,t)$, defined by $z=1,
u=1$. Hence it is enough to construct a resolution for the restriction
$U$ of $Y_0$ to this affine chart:
\begin{equation*}
U \;\; : \;\; xy + xt^2 + ys^2 = 0
\end{equation*} 
By making the change of variables $y_1=y+t^2$ and $x_1=x+s^2$ we
obtain the equation:
\begin{equation*}
U \;\; : \;\; x_1y_1-s^2t^2=0
\end{equation*} 
The singular locus of $U$ is thus the union of two curves:
\begin{align*}
\begin{split}
B_s & {}\;\; : \;\; x_1=y_1=s=0 \\ 
B_t & {}\;\; : \;\; x_1=y_1=t=0
\end{split}
\end{align*}

Let $\tilde U\to U$ be the composition of the blow up along $B_s$ and
then the blow up along the strict transform of $B_t$. We claim that
$\tilde U$ is smooth and that the map $\tilde U\to U$ is universally
$\CH_0$-trivial. To check this, we use the local blow up calculations
below. Note that we need to discuss only two charts in each case,
using symmetry between $x_1$ and $y_1$, and $x_2$ and $y_2$, with the
notations below.

\begin{enumerate}
\item[\textit{1)}] First blow up $U$ along $B_s$:
\begin{itemize}
\item In the chart defined by $y_1=x_1y_2, s=x_1s_2$, the equation of the blow up is $y_2-t^2s_2^2=0$, and the exceptional divisor is $x_1=0, y_2-t^2s_2^2=0$.
\item In the chart defined by $x_1=sx_2, y_1=sy_2$, the equation of the blow up is $x_2y_2-t^2=0$ and the  the exceptional divisor is $s=0, x_2y_2-t^2=0$. 
\end{itemize}

\item[\textit{2)}] Second blow up the proper transform $B_t'$ of $B_t$: 
$$
B_t' \;\; : \;\; x_2=y_2=t=0.
$$
\begin{itemize}
\item In the chart defined by $y_2=x_2y_3, t=x_2t_3$, the equation of the blow up is $y_3-t_3^2=0$, the exceptional divisor is $x_2=0, y_3-t_3^2=0$.
\item In the chart defined by $x_2=tx_3, y_2=ty_3$, the equation of the blow up is $x_3y_3-1=0$, the exceptional divisor is $t=0, x_3y_3-1=0$.
\end{itemize}
\end{enumerate}
We see immediately that $\tilde U$ is smooth, and that the resolution
$\tilde U \to U$ has scheme-theoretic fibers that are either smooth
rational conics or chains of lines, hence these fibers are universally
$\CH_0$-trivial.  We conclude, using \cite[Prop.~1.8]{CTP}, that
$\tilde U \to U$ is a universally $\CH_0$-trivial resolution. 

Now we verify the rationality of $Y_0$.  As a divisor of bidegree
$(2,2)$ in $\P^2 \times \P^2$, the variety $Y_0$ admits a conic bundle
structure for each of the projections to $\P^2$.  From the explicit
representation \eqref{eq:HPT_conic}, we see that under the projection
to the $\P^2$ with homogeneous coordinates $(u:t:s)$, the conic bundle
$Y_0 \to \P^2$ is defined by the matrix of homogeneous forms:
\begin{equation}\label{eq:conicY_0}
\begin{pmatrix}
0 & u^2 & t^2 \\
u^2 & 0 & s^2 \\
t^2 & s^2 & 0
\end{pmatrix}
\end{equation}
Since the quadratic form is clearly isotropic over the generic point
of the base $\P^2$ (for example, $(x:y:z)=(1:0:0)$ is an isotropic
vector), the generic fiber of $Y_0$ is rational over a rational base,
hence $Y_0$ is rational.
\end{proof}

We can also analyze the singularities of $Y_0$ from the point of view
of the conic bundle structure $Y_0 \to \P^2$ considered in
\eqref{eq:conicY_0}.  Note that the discriminant is the double
triangle $(uts)^2$ in $\P^2$ and the conic fibers have constant rank 2
along the discriminant.  The singular locus of $Y_0$ consists of a
rational curve above each line of the double triangle forming the
discriminant.  We now describe the analytic local normal forms for
such a conic bundle.

Let $k$ be an algebraically closed field of characteristic $\neq
2$. Considering a point on only one of the double lines, we have the
following.  Any conic over a complete 2-dimensional regular local ring
$k[[u,t]]$ degenerating to conics of rank 2 over $u^2$ can be brought
to the normal form
$$
x^2 + y^2 + u^2 z^2 = 0.
$$
In this case, the singular locus is the rational curve $x=y=u=0$ lying
over $u=0$ in the base. Simply blowing up this curve is a universally
$\CH_0$-trivial resolution.

Now considering a point on the intersection of two of the double
lines, we have the following.  Any conic over a complete 2-dimensional
regular local ring $k[[u,t]]$ degenerating to conics of rank 2 over
$(uv)^2=0$ can be brought to the normal form
$$
x^2 + y^2 + u^2v^2 z^2 = 0.
$$
In this case, the singular locus is the union of rational curves
$x=y=u=0$ and $x=y=v=0$ lying over $uv=0$ in the base. Blowing up one
of these rational curves, and then the strict transform of the other
yields a universally $\CH_0$-trivial resolution whose fiber above the
generic point of either $u=0$ or $v=0$ is a $\P^1$ (over the function
field of a curve over $k$) and above the intersection point
is a chain of three smooth rational curves.

\smallskip

Finally, we combine all this together to give our first proof of
Theorem~\ref{thm:main_intro}.

\begin{proof}[Proof of Theorem~\ref{thm:main_intro}]
Let $Y_1$ be the reference variety \eqref{refHPT} for hypersurfaces of
bidegree $(2,2)$ in $\P^2 \times \P^3$.  The discriminant of the conic
bundle $Y_1 \to \P^3$ defined by projection to the second factor is
the union of two Cayley cubic surfaces meeting along a triangle of
lines and a configuration of three smooth conics.  Let $H \subset
\P^3$ be the unique hyperplane through this triangle of lines.  Then
$Y_2 = \P^2 \times H$ is a hypersurface of bidegree $(0,1)$ in $\P^2
\times \P^3$.  Consider the reducible projective variety $Y = Y_1 \cup
Y_2$, which is a hypersurface of bidegree $(2,3)$ in $\P^2 \times
\P^3$.  Then $Y_0 = Y_1 \cap Y_2$ is the irreducible projective
variety with equation \eqref{eq:HPT_section}, which by
Proposition~\ref{prop:resolution_HPT_section}, admits a universally
$\CH_0$-trivial resolution $\tilde{Y}_0 \to Y_0$ with $\tilde{Y}_0$
smooth and rational, hence in particular, $Y_0$ is universally
$\CH_0$-trivial.  Thus by an application of
Proposition~\ref{prop:resolution_HPT_section}, we have that $Y$ is a
reference variety for hypersurfaces of bidegree $(2,3)$ in $\P^2
\times \P^3$.  By the specialization method \cite{voisin}, \cite{CTP},
then the very general hypersurface of bidegree $(2,3)$ in $\P^2 \times
\P^3$ is not universally $\CH_0$-trivial, and in particular, is not
stably rational.
\end{proof}



\begin{rema}\label{rem:somerational}
We remark that in the Hilbert scheme of all hypersurfaces of bidegree
$(2,3)$ in $\P^2 \times \P^3$, there is a nonempty closed subvariety
whose general element $Y$ is rational.  Indeed, any bidegree $(2,3)$
hypersurface $Y$ in $\P^2 \times \P^3$ of the form
$$
A sv + B su + C tv + D tu = 0,
$$
where $A,B,C,D$ are general homogeneous forms of bidegree $(2,1)$ on
$\P^2 \times \P^3$, is smooth and the generic fiber of the associated
cubic surface bundle $Y \to \P^2$ contains the disjoint lines
$\{s=t=0\}$ and $\{u=v=0\}$, hence is a rational cubic surface over
$k(\P^2)$, hence $Y$ is a rational variety.  This provides a simple
and geometrically appealing of a family of smooth projective
rationally connected fourfolds whose very general fiber is not stably
rational but where some fibers are rational, cf.\ \cite{HPT}.
\end{rema}

\section{Remarks on the Brauer group of cubic surface bundles}
\label{sec:cubic}

In this section, we study the Brauer group of a cubic surface bundle
over a rational surface.  Cubic surface bundles naturally arise in the
study of hypersurfaces of bidegree $(2,3)$ in $\P^2 \times \P^3$, via
projection onto the first factor.

Let $S$ be a smooth projective rational surface over a field $k$ of
characteristic not dividing 6 and let $K = k(S)$.  Let $\pi : X \to S$
be a flat projective morphism whose generic fiber $X_K$ is a smooth
cubic surface over $K$.  (We will also temporarily consider cases when
the generic fiber is not smooth.)  Moreover, we will assume that
$\pi_* \omega_{X/S}\dual = E$ is a rank 4 locally free $\OO_S$-module
and that $X \subset \P(E)$ is a relative cubic hypersurface over $S$
defined by the vanishing of a global section of $S^3(E\dual)\tensor L$
for some line bundle $L$ on $S$.

The locus in $S$ over which the fibers are singular is a divisor that
carries a canonical scheme structure called the \linedef{discriminant}
divisor $\Delta \subset S$ of the cubic surface bundle.  This divisor
can be constructed using invariant theory as follows.

Consider the Hilbert scheme 
$$
\HH = \P(H^0(\P^3,\OO(3))) = \P^{19}
$$
of cubic surfaces in $\P^3$ over $k$ and the action of $\SL_4$ on
$\HH$ by change of variables.  Salmon~\cite{salmon:discriminant}, and
independently Clebsch~\cite{clebsch:quaternary_cubic},
\cite{clebsch:quaternary_cubic_II}, found that the ring of invariants
of cubic forms in four variables is generated by \linedef{fundamental
invariants} $A,B,C,D,E$ in degrees 8, 16, 24, 32, 40, and 100 where
the square of the invariant of degree 100 is a polynomial in the
remaining invariants.  This implies that the associated GIT quotient
is isomorphic to a weighted projective space
$$
\P(H^0(\P^3,\OO(3)))/\!\!/\mathrm{SL}_4 \isom \P(8,16,24,32,40)\isom
\P(1,2,3,4,5).
$$
Salmon found a formula for the discriminant in terms of the
fundamental invariants, though his formula contained an error.  This
error had been repeated throughout the 19th and 20th century until it
was corrected by Edge~\cite{edge:discriminant_cubic_surface}
$$
\Delta = (A^2 - 64 B)^2 - 2^{11}(8 D + AC) 
$$
For further modern study of the fundamental invariants and the
discriminant of a cubic surface, see
\cite{beklemishev:cubic_four_variables},
\cite[\S10.4]{dolgachev:lectures_invariant_theory},
\cite[\S2]{elsenhans_jahnel:discriminant_cubic_surface}.

The formula for the discriminant $\Delta$ in terms of the monomials of
a generic cubic surface defines an integral hypersurface of degree 32
in $\HH$.  Given a quadric surface bundle $\pi : X \to S$, with $X
\subset \P(E)$ for a rank 4 vector bundle $E$ on $S$, and a Zariski
open $U \subset S$ over which $E$ is trivial, the restricted cubic
surface bundle $\pi|_{U} : X|_U \to U$ is defined by the restriction
of the universal cubic hypersurface on $\HH$ via a classifying map $U
\to \HH$.  Then the locus of points in $U$ over which the fibers of
$\pi|_U$ are singular is contained in the divisor $\Delta|_U \subset
U$.  By choosing a Zariski open cover of $S$ trivializing $E$, we can
glue these divisors to yield the discriminant divisor $\Delta \subset
S$ of the cubic surface bundle $\pi : X \to S$.

\medskip

In the rest of the section, we make a series of remarks about the
second unramified cohomology group $\Hur^2(k(X)/k,\Q/\Z(1))$ of the
total space of a cubic surface bundle $\pi : X \to S$, or what is the
same, the Brauer group of a smooth proper model of $X$.  We will often
call this group the \linedef{unramified Brauer group} $\Brur(k(X)/k)$.

We have $\Brur(k(X)/k) \subset \Brur(k(X)/K) = \Br(X_K)$, where the
second equality holds by purity (cf.\
\cite[\S2.2.2]{colliot:santa_barbara}) because the generic fiber
$X_K$ is a smooth projective cubic surface over $K$. Part of the
sequence of low degree terms of the Leray spectral sequence for the
\'etale sheaf $\Gm$ associated to the structural morphism $X_K \to
\Spec(K)$ is
\begin{equation}\label{eq:leray}
\small
\Pic(X) \to H^0(K,\Pic(X_{K\sep})) \to \Br(K) \to \Br(X_K) \to H^1(K,\Pic(X_{K\sep})) \to H^3(K,\Gm)
\end{equation}
where here $\Br(X_K) = \ker(\Br(X_K) \to \Br(X_{K\sep}))$ because
$X_K$ is geometrically rational.  Hence the cokernel of the map
$\Br(K) \to \Br(X_K)$ is isomorphic to a subgroup of
$H^1(K,\Pic(X_{K\sep}))$.  In fact,
Swinnerton-Dyer~\cite{swinnerton-dyer:Brauer_group_cubic_surface} has
computed the possible nontrivial values for $H^1(K,\Pic(X_{K\sep}))$
when $X_K$ is a smooth cubic surface: they are $\Z/2$, $\Z/2 \times
\Z/2$, $\Z/3$, and $\Z/3 \times \Z/3$.  As a corollary, we arrive at
the following.

\begin{prop}
\label{prop:Br6}
Let $X$ be a smooth cubic surface over a field $K$.  Then the cokernel
of the map $\Br(K) \to \Br(X)$ is killed by 6.
\end{prop}

We remark that there is a general method to obtain bounds on the
torsion exponent of the Brauer group of a smooth proper variety $X$
over a field $K$ from torsion exponent bounds on the Chow group
$A_0(X)$ of 0-cycles of degree 0.  By a generalization of an argument
of Merkurjev, see
\cite[Thm.~1.4]{ACTP:unramified_cohomology_cubic_fourfold}, if
$A_0(X)$ is universally $N$-torsion (this is related to the
\linedef{torsion order} of $X$, cf.\ \cite{chatzistamatiou_levine})
then the cokernel of the map $\Br(K) \to \Br(X)$ is killed by $N\,
I(X)$, where $I(X)$ is the index of $X$, the minimal degree of a
0-cycle.  For a smooth projective cubic surface $X_K$ over a field
$K$, the Chow group $A_0(X_K)$ of 0-cycles of degree 0 is universally
6 torsion by an argument going back to Roitman, see
\cite[Prop.~4.1]{chatzistamatiou_levine}.  The index of a cubic
hypersurface always divides 3, as one can see by cutting with a line.
Thus, in the case of cubic surfaces, the Galois cohomology
computations of Swinnerton-Dyer achieve a better bound than that which
can be obtained by the method involving 0-cycles.

Given a smooth cubic surface $X$ over a field $K$, we can also bound
the kernel of the map $\Br(K) \to \Br(X)$, otherwise known as the
\linedef{relative Brauer group} $\Br(X/K)$.  The following must be
well known to the experts, but we could not find a precise reference
in the literature.

\begin{prop}
\label{prop:Brauer_pullback}
Let $X$ be a smooth proper $K$-variety.  Then the kernel of the map
$\Br(K) \to \Br(X)$ is killed by the index $I(X)$. 
\end{prop}
\begin{proof}
Let $x \in X$ be a closed point with residue field $L/K$.  We recall
the existence of a corestriction homomorphism $\cor_{L/K} : \Br(L) \to
\Br(K)$ with the property that the composition $\Br(K) \to \Br(L) \to
\Br(K)$ is multiplication by the degree $[L:K]$.  The
algebra-theoretic norm (determinant of the left-regular
representation) yields a norm homomorphism $n_{L/K} : R_{L/K}\Gm \to
\Gm$ of group schemes over $K$.  The corestriction homomorphism is
then the composition
$$
\Br(L) = H^2(L,\Gm) = H^2(K,R_{L/K}\Gm) \mapto{H^2(n_{L/K})}
H^2(K,\Gm) = \Br(K).
$$
When $L/K$ is separable, this coincides with the classical
corestriction map on Galois cohomology; when $L/K$ is purely
inseparable, this is simply multiplication by the degree.  Then define a map
$\Br(X) \to \Br(L) \to \Br(K)$ by restriction to the point $x \in X$
followed by corestriction.  This map has the property that the
composition $\Br(K) \to \Br(X) \to \Br(K)$ is multiplication by the
degree of the point~$x$. In particular, the kernel of the map $\Br(K)
\to \Br(X)$ is killed by the degree of~$x$.  Since the index $I(X)$
agrees with the greatest common divisor of the degrees of closed
points, we arrive at the claim.
\end{proof}

We remark that given a 0-cycle $z = \sum_i a_i z_i$ of minimal degree
$I(X)$, we can define (in analogy with the proof of
Proposition~\ref{prop:Brauer_pullback}) a map $\Br(X) \to \Br(K)$ by
$\alpha \mapsto \sum_i a_i\, \cor_{K(z_i)/K}(\alpha|_{z_i})$, where
$\alpha|_{z_i}$ is the restriction of the Brauer class to the closed
point $z_i$.  This map has the property that the composition $\Br(K)
\to \Br(X) \to \Br(K)$ is multiplication by the index $I(X)$.
Finally, answering a question of Lang and Tate from the late 1960s, it
is a result of Gabber, Liu, and Lorenzini~\cite{gabber_liu_lorenzini:index}
that there always exists a 0-cycle of minimal degree on a smooth
proper variety $X$ whose support is on points with separable residue
fields.

Note that since a smooth cubic surface always has index dividing 3 (by
cutting with a line), we see that the relative Brauer group $\Br(X/K)$
is always killed by 3.

We remark that the relative Brauer group of a smooth cubic surface can
be nontrivial.  Indeed, if $X$ admits a Galois invariant set of 6
non-intersecting exceptional curves, then $X$ is the blow up of a
Severi--Brauer surface along a point of degree 6.  If this
Severi--Brauer surface has nontrivial Brauer class $\alpha \in
\Br(K)[3]$ then the relative Brauer group $\Br(X/K)$ is generated by
$\alpha$.  In general, for a smooth cubic surface $X$ over $K$, by
\eqref{eq:leray}, the relative Brauer group
$\Br(X/K)$ is isomorphic to $H^0(K,\Pic(X_{K\sep}))/\Pic(X)$, hence is
a finite elementary abelian $3$-group.

We also remark that the bound in
Proposition~\ref{prop:Brauer_pullback} is not sharp in general.
Indeed, let $Q$ be a smooth projective quadric surface with nontrivial
discriminant (i.e., Picard rank 1) and no rational point over a field
$K$ of characteristic not 2.  Then $I(Q)=2$ yet the map $\Br(K) \to
\Br(Q)$ is injective, cf.\
\cite[Thm.~3.1]{ACTP:unramified_cohomology_cubic_fourfold}.

Now we will consider the subgroup $\im(\Br(K) \to \Br(k(X)) \cap
\Brur(k(X)/k)$ using a method going back to Colliot-Th\'el\`ene and
Ojanguren~\cite{colliot-thelene_ojanguren:Artin-Mumford}.  Consider
the following commutative diagram, for $\ell$ a prime different from
the characteristic:
{\small
\begin{gather}
\label{diag:basic}
\xymatrix@R=15pt{
0 \ar[r] & \Hur^2 (k (X)/X , \mu_\ell) \ar[r] & \Hur^2(k (X)/K, \mu_\ell) 
\ar[r]^{\oplus \partial^2_{T} \quad\quad} & \bigoplus_{T \in X^{(1)}_S} H^1 (k (T), \Z /\ell) \\
            & 0=\Hur^2(K/k,\mu_\ell) \ar[r] & H^2 (K, \mu_\ell)
            \ar[r]^(.42){\oplus \partial^2_C}  \ar[u]^(.45){\iota }
            & \bigoplus_{C \in S^{(1)}} H^1 (k (C), \Z
            /\ell)\ar[u]^(.38){\tau}
}
\end{gather}
}

We first explain the new pieces of notation: $\Hur^2 (k (X)/X ,
\mu_\ell)$ denotes all those classes in $H^2 (k(X), \mu_\ell)$ that
are unramified with respect to divisorial valuations corresponding to
prime divisors (integral threefolds) on $X$.
Moreover, $\Hur^2 (k (X)/K, \mu_\ell)$ is the subset of those
classes in $H^2 (k(X), \mu_\ell)$ that are unramified with respect to
divisorial valuations that are trivial on $K$, hence correspond to
prime divisors of $X$ dominating the base $S$.

In the upper row, $T$ runs over all integral threefolds, i.e.,
prime divisors, in $X$ that do not dominate the base $S$, hence map to
some curve in $S$. We call this set of integral threefolds
$X^{(1)}_S$. Then the upper row is exact by definition.

In the lower row, $C$ runs over the set $S^{(1)}$ of all integral
curves in $S$.  Thus this row coincides with the usual Bloch--Ogus
complex for degree 2 \'etale cohomology associated to $S$.  The $i$th
cohomology group of this complex is computed by the Zariski cohomology
$H^i(S,\HH^2)$ of the sheaf $\HH^i$, which is the sheafification of
the Zariski presheaf $U \mapsto \Het^2(U,\mu_\ell)$, see
\cite[Thm.~6.1]{bloch_ogus}.  In particular, the lower row is exact in
the first two places because $H^0(S,\HH^2) = \Br(S)[\ell] =0$ and
$H^1(S,\HH^2) \subset \Het^3(S,\Z/\ell)=0$ since $S$ is a smooth
projective rational surface, where the later inclusion arises from the
sequence of low terms associated to the Bloch--Ogus spectral sequence
$H^i(S,\HH^j) \Rightarrow
\Het^{i+j}(S,\mu_\ell)$.

Now we discuss the vertical arrows.  The map $\iota$ is the usual
restriction map in Galois cohomology associated to the field extension
$k(X)/K$.  When $X|_C$ is generically reduced over $C$ and $T \subset
X|_C$ is a component, the map $\tau$ is the usual restriction map in
Galois cohomology for the field extensions $k(T)/k(C)$.
When $T$ is the reduced subscheme of a component of $X|_C$, then the
map $\tau$ involves multiplication by the multiplicity.

If we want to understand classes $\xi \in H^2(K,\mu_\ell)$ such that
$\iota(\xi) \in \Hur^2(k(X)/X,\mu_\ell)$, then $\partial_T^2\iota(\xi)
= 0$ for all $T$ as in the diagram.  By the commutativity of the
central square in the diagram, we see that then $\tau(\partial^2_C
\xi)=0$ for all integral curves $C$ in $S$.  Thus we are interested in
the kernel $\ker(\tau)$, which we call $\KK_\ell$.  Clearly,
$\KK_\ell$ is the direct sum of the kernels of all restriction maps
$\tau : H^1(k(C),\Z/\ell) \to H^1(k(T),\Z/\ell)$.  The following lemma
becomes relevant.

\begin{lemm}\label{lem:pullback}
Let $Z$ be a proper integral $k$-variety with function
field $K$.  Let $l$ be the separable closure of $k$ inside $K$, i.e.,
the field extension of $k$ determined by the finite Galois set of
irreducible components of $Z \times_k k\sep$.  Then the kernel of the
restriction map $H^1(k,\Z/\ell) \to H^1(K,\Z/\ell)$ coincides with the
kernel of the restriction map $H^1(k,\Z/\ell) \to H^1(l,\Z/\ell)$.  In
particular, if $Z$ is geometrically integral, then the restriction
map $H^1(k,\Z/\ell) \to H^1(K,\Z/\ell)$ is injective.
\end{lemm}


Recall that since the generic fiber $X_K$ of the cubic surface
fibration $\pi : X \to S$ is smooth, the locus in $S$ over which the
fibers become singular is the support of a divisor, the discriminant
divisor $\Delta \subset S$.  Thus if $C$ is not contained in the
discriminant divisor, then $T=X|_C \to C$ is a generically smooth
cubic surface fibration (in particular, there is a unique codimension
1 point of $X$ above the generic point of $C$), and
Lemma~\ref{lem:pullback} implies that $\tau$ is injective on the part
of the direct sum corresponding to $T\to C$.  Hence, we are only
interested in components of the discriminant divisor $\Delta = \cup
\Delta_i$.  However, for a general cubic surface fibration, we expect
that the generic fiber above a reduced component of the discriminant
has only isolated singularities, in particular is integral.  

Let $C$ be a component of the discriminant.  When $T$ is the reduced
subscheme of a component of $X|_C$ with multiplicity $e >1$, then the
map $\tau$ in diagram \eqref{diag:basic} will be zero if $\ell$
divides $e$.  Now assume that the generic fiber of $X|_C \to C$ is
reduced and admits an irreducible component that is not geometrically
integral. Considering the possible degenerations of a cubic surface in
$\P^3$, we see that the only possible cases are when the geometric
generic fiber above $C$ is a union of three planes or plane and an
irreducible quadric.  In the first case, the finite Galois set of
connected components of the geometric generic fiber of $X|_C \to C$ is
isomorphic to the spectrum of an \'etale $k(C)$-algebra of degree 3.

\begin{lemm}\label{lem:H-S}
Let $K$ be a field and $L/K$ a finite separable extension of degree
$n=2,3$. Let $\ell$ be a prime 
and write $\KK_\ell = \ker(H^1(K,\Z/\ell) \to H^1(L,\Z/\ell))$.  Then
$\KK_\ell=0$ whenever $\ell \neq n$ or when $\ell=n=3$ and $L/K$ is not
Galois.  Otherwise, $\KK_\ell$ is generated by the class $[L/K] \in
H^1(K,\Z/\ell)$, when $n=\ell$ and $L/K$ is cyclic.
\end{lemm}
\begin{proof}
If $L/K$ is a $G$-Galois extension of fields, then the sequence of low
degree terms of the Hochschild--Serre spectral sequence implies that
the kernel of the restriction map $H^1(K,\Z/\ell) \to H^1(L,\Z/\ell)$
is isomorphic to the group cohomology $H^1(G,\Z/\ell)$.  When
$n=\ell$ and $L/K$ is cyclic, this shows that the kernel of the
restriction is cyclic of order $\ell$, and since we already know that
the class $[L/K]$ becomes trivial, we are done.

Now consider the case when $L/K$ is cubic and not Galois, in which
case the normal closure $N/K$ of $L/K$ is an $S_3$-Galois extension.
Hence the kernel of the composition of restriction maps
$H^1(K,\Z/\ell) \to H^1(L,\Z/\ell) \to H^1(N,\Z/\ell)$ is isomorphic
to the group cohomology $H^1(S_3,\Z/\ell)$, which has order 2 when
$\ell=2$ and is trivial for all primes $\ell \neq 2$.  In the case
$\ell=2$, we know that $N/L$ is degree 2 so that the kernel of the
restriction map $H^1(L,\Z/2) \to H^1(N,\Z/2)$ already has order 2.
Hence the map $H^1(K,\Z/2) \to H^1(L,\Z/2)$ is injective.
\end{proof}

We combine these considerations to arrive at a sufficient condition
for the triviality of $\im(\Br(K) \to \Br(k(X))) \cap \Brur(k(X)/k)$.

\begin{theo}
\label{thm:LST}
Let $\pi : X \to S$ be a flat cubic surface bundle over a smooth
projective connected rational surface $S$ over an algebraically closed
field $k$ of characteristic not dividing 6.  Assume that the generic
fibers of $\pi$ over irreducible curves $C \subset S$ are reduced.
Then $\im(\Br(K)[2] \to \Br(k(X)) \cap \Brur(k(X)/k)$ is trivial.

If we furthermore assume that the generic fibers of $\pi$ over
irreducible curves $C \subset S$ are never geometrically the union of
three planes permuted cyclically by Galois, then $\im(\Br(K)[3] \to
\Br(k(X)) \cap \Brur(k(X)/k)$ is trivial.
\end{theo}
\begin{proof}
Let $\alpha \in \Br(K)$ be a nontrivial element, then there is an
irreducible curve $C \subset S$ such that the residue of $\alpha$ at
$C$ is nontrivial.  We argue that there is always an irreducible
component $T \subset X|_C$ of the fiber of $\pi$ above $C$ such that
the map $\tau : H^1(k(C),\Z/\ell) \to H^1(k(T),\Z/\ell)$, for either
$\ell=2$ in the first case or $\ell=3$ in the second case, is
injective.  Hence $\iota(\alpha)$ will have a nontrivial residue at a
valuation of $k(X)$ centered on $k(T)$, so $\iota(\alpha)$ is not in
$\Brur(k(X)/k)$.  Indeed, if the generic fiber of $X|_C \to C$ is
geometrically irreducible, the injectivity follows from
Lemma~\ref{lem:pullback}; if this generic fiber is not geometrically
irreducible, then by the discussion above, there are only the
following possible cases.  Either the generic fiber of $X|_C \to C$
has an irreducible component that is a plane; in this case, we take
$T$ equal to this component and again apply Lemma~\ref{lem:pullback}.
Or, the generic fiber of $X|_C \to C$ is irreducible and is
geometrically the union of three planes; then by assumption, we are in
the situation of Lemmas~\ref{lem:pullback} and~\ref{lem:H-S}, and we
can take $T = X|_C$.
\end{proof}

For an application of Theorem~\ref{thm:LST}, see Remark~\ref{rem:nonconstant}.

\section{Quadric surface bundles with polygonal discriminant}
\label{sec:poly}

In this section, we construct new reference quadric surface bundle
fourfolds.

\subsection{Polygonal discriminant}

Let $C\subset \P^2$ be a conic and let $P_{m}$ be a regular $m$-gon,
for $m\geq 3$, such that $C$ is inscribed in $P_{m}$, all defined over
an algebraically closed field $k$ of characteristic zero. Let $L_1,
\dotsc, L_m\subset \P^2$ be the lines corresponding to the sides of the
$m$-gon and let $\ell_1, \dotsc, \ell_m$ be the linear forms defining
these lines. We assume that the conic is defined by
$$F(x,y,z):=x^2+y^2+z^2-2xy-2xz-2yz=0,$$
to make the notations consistent with  \cite{HPT, HPT-double}. 
Let $a,b,c,d$ be the arbitrary products of some of the linear forms $\ell_1, \dotsc, \ell_m$, satisfying the following  conditions:
\begin{align}\label{conditions}
\begin{split}
& abcd=(\ell_1\dotsm\ell_m)^2,\\ 
& \mbox{the degrees of }a,b,c,d  \mbox{ are either all even or all odd,}\\
& a,b,c,d\mbox{ are square free and no two are equal.}\\
\end{split}
\end{align}
Put $n=m+1$ and consider  $X_n\to \P^2$ defined by
\begin{equation}\label{eproj}
as^2+bt^2+cu^2+dFv^2=0.
\end{equation}
Then $X_n$ is a quadric surface bundle, in a weak sense (i.e., no
flatness condition is assumed, see \cite{Sch1,Sch2}), that is a
relative hypersurface in a projective bundle over $\P^2$.  Note that
by this construction we obtain all possible degrees
$d_0=\mathrm{deg}\,a, d_1=\mathrm{deg}\,b, d_2=\mathrm{deg}\,c,
d_3=\mathrm{deg}\,dF$, either all even or odd, at most one degree is
zero, with the following restriction
\begin{align}\label{restriction}
\begin{split}
&d_0+d_1+d_2+d_3=2n,\\ 
&d_0\leq n-1,\; d_1\leq n-1,\; d_2\leq n-1,\; 2\leq d_3\leq n+1.
\end{split}
\end{align}

\subsection{Second unramified cohomology group}

\begin{theo}\label{obs}
Let $X_n$ be defined as in \eqref{eproj}. Then 
$$
\Hur^2(k(X_n)/k, \Z/2)\neq 0.
$$
\end{theo}
\begin{proof}
We apply \cite[Theorem 3.17]{pirutka-survol}. The generic fiber of the quadric bundle $X_n$ is given by a diagonal quadratic form $q\simeq \langle a,b,c,dF \rangle$. The Clifford invariant of $q$ is given by
$$\alpha=c(q)=(a,b)+(c,dF)+(ab, cdF).$$

From the construction, since the conic $C$ is tangent to the lines $L_1, \dotsc, L_m$, we deduce that the residue $\partial_C^2(\alpha)$ at the generic point of $C$ is trivial.  Hence  the ramification divisor  $ram\,\alpha$ is supported on the union of lines $L_1\cup\dotsm\cup L_m$, which is a simple normal crossings divisor.   Hence {\it loc. cit.}  applies and we deduce that $\Hur^2(k(X_n)/k, \Z/2)\neq 0$. 
More precisely, the image $\alpha'$ of $\alpha$ in $H^2(k(X_n), \Z/2)$  is a nontrivial element of the group $\Hur^2(k(X_n)/k, \Z/2)$:  with the notations in \cite[Theorem 3.17]{pirutka-survol}, we have $T=ram\, \alpha$, and the diagonal class $(1,\dotsc, 1)$, corresponding to the set of residues of $\alpha'$ gives an element  of the group $H$.
\end{proof}

\section{A cubic surface bundle with nontrivial unramified Brauer group}
\label{sec:cubic_example}

In this section, we construct an irreducible reference hypersurface
$Y$ of bidegree $(2,3)$ in $\P^2 \times \P^3$ over an algebraically
closed field $k$ of characteristic zero, thereby giving another proof
of Theorem~\ref{thm:main_intro}.  As a consequence, we also arrive at
an explicit example of a smooth cubic surface $X$ over $K=k(\P^2)$
such that the cokernel of the map $\Br(K) \to \Br(X)$ contains a
nontrivial 2-torsion class that is unramified on a fourfold model of
$X$ over $k$q.

Consider the following varieties:
\begin{itemize}
\item The weak quadric surface bundle $V\to \P^2$ with
polygonal discriminant defined by
\begin{equation}\label{qb}
u(u-v)(t-v)x^2+u(u-v)ty^2+tw^2+(t-v)F(u,t,v)z^2=0
\end{equation}
where $\P^2$ has homogeneous coordinates $(u : v : t)$, and where
$F(u,t,v)=0$ is the conic inscribed in the square formed by four
linear forms $\ell_1=u, \ell_2=u-v, \ell_3=t, \ell_4=t-v$. Explicitly,
we have
\begin{equation*}
F(u,t,v)=4u^2+4t^2+v^2-4uv-4tv;
\end{equation*}

\item The hypersurface $Y$ of bidegree $(2,3)$ in $\P^2\times \P^3$
defined by
\begin{equation}\label{bd23}
u(u-v)(t-v)x^2+u(u-v)ty^2+tw^2z^2+(t-v)F(u,t,v)z^2=0.
\end{equation}
where here, $\P^2$ has homogeneous coordinates $(x : y : z)$ and
$\P^3$ has homogeneous coordinates $(u : t : v: w)$.

\end{itemize}

Note that $V$ is birational to $Y$ since the open subset of $V$ defined by $z=1, u=1$ and the open subset of $Y$ defined by $z=1, u=1$ are given by the same affine equation in variables $v,t,w,x,y$.  

We need the following technical result, which will be proved in
\S\ref{subsec:birat}.

\begin{prop}\label{birtransform}
Let $V$ and $Y$ be defined as in \eqref{qb} and \eqref{bd23},
respectively.  Then there is a proper birational morphism $\pi:Y'\to
Y$ such that:
\begin{enumerate}
\item\label{prop:i} the rational map $Y\dashrightarrow V$ extends to a morphism $f:Y'\to V$;
\item\label{prop:ii} the maps $\pi$ and $f$ are universally $\CH_0$-trivial.
\end{enumerate}
\end{prop}

With this result, we are ready to give our second proof of
Theorem~\ref{thm:main_intro}.

\begin{proof}[Proof of Theorem~\ref{thm:main_intro}]
By the specialization method \cite{voisin, CTP}, it is enough to
insure that $Y$ is a reference variety, i.e., that the conditions (O)
and (R) as in the introduction are satisfied for $Y$.  Since $Y$ is
birational to $V$, we have by Theorem~\ref{obs} that $\Hnr^2(k(Y)/k,
\Z/2)=\Hnr^2(k(V)/k, \Z/2)\neq 0$. Let $Y'$ be as in
Proposition~\ref{birtransform}. By Theorem~\ref{resolution} below,
there exists a $\CH_0$-trivial resolution $\tilde V \to V$. Let
$\tilde Y \to Y'$ be a resolution such that the rational map
$Y'\dashrightarrow \tilde V$ extends to a morphism $\tilde Y\to \tilde
V$. Since $\tilde Y$ and $\tilde V$ are smooth, the map $\tilde Y\to
\tilde V$ is universally $\CH_0$-trivial. This follows from weak
factorization and the fact that the blow up of a smooth variety along
a smooth center is a universally $\CH_0$-trivial morphism. We then
have the following diagram
$$
\xymatrix@R=14pt{
\tilde Y\ar[d]\ar[dr]&  \\
Y'\ar[d]\ar[dr]& \tilde V\ar[d]\\
Y& V,
}
$$ 
where we know that all the maps, except possibly $\tilde Y\to Y'$, are
universally $\CH_0$-trivial. From this diagram, we see that the map
$\tilde Y\to Y'$ is universally $\CH_0$-trivial, hence the composition
$\tilde Y\to Y'\to Y$ is a universally $\CH_0$-trivial resolution.
Finally, $Y$ is a reference variety.
\end{proof}

\begin{rema} Note that if $\phi$ is the composition $\phi:Y\dashrightarrow V\to \P^2$ and $U\subset Y$ is an open where $\phi$ is defined, then the  generic fiber of the map $\phi : U\to \P^2$ is not proper, so that we could not directly apply \cite[Theorem 9]{Sch1}.
\end{rema}

\begin{rema}
\label{rem:nonconstant}
A nice feature of the reference variety $Y$ is that the cubic surface
bundle $Y \to \P^2$ has as generic fiber $X=Y_K$ an explicit example
of a smooth cubic surface over $K=k(\P^2)$ with a nontrivial element
$\alpha \in \Br(X)[2]\neq 0$ that is not contained in the image of the
map $\Br(K) \to \Br(X)$ and $\alpha \in \Brur(k(X)/k)$ is globally
unramified on the function field of the fourfold $Y$.  The fact that
$\alpha$ is unramified follows from Theorem~\ref{obs}.  To prove that
$\alpha$ is nonconstant, we explicitly compute that the discriminant
of the cubic surface fibration $Y \to \P^2$ is the union of the line
$\{x=0\}$ with multiplicity 4, the line $\{y=0\}$ with multiplicity 4,
the line $\{z=0\}$ with multiplicity 30, the pair of conjugate lines
$\{x^2+z^2=0\}$ with multiplicity 6, the smooth conic
$\{x^2-y^2+z^2=0\}$ with multiplicity 4, and the integral sextic
defined by
$$
x^6 + 3 x^4 y^2 + 3 x^2 y^4 + y^6 + x^4 z^2 - 8 y^4 z^2 + 16 y^2 z^4 +
20 x^2 y^2 z^2 = 0.
$$
%
For example, one could use Magma's {\tt ClebschSalmonInvariants}
command \cite{magma}. Then, over each component besides $\{z=0\}$,
one can check that the generic fiber is geometrically integral, while
over the component $\{z=0\}$, the fiber is the union of three planes
defined over the residue field of the line $\{z=0\}$, hence
Theorem~\ref{thm:LST} applies and no constant Brauer class is
unramified on $X$.
\end{rema}

\subsection{Birational transformation}
\label{subsec:birat}

In this section, we prove Proposition \ref{birtransform}.  We have the
following natural rational map $Y\dashrightarrow V$
\begin{equation}\label{1map}
(u:v:t:w, x:y:z)\mapsto (u:v:t, x:y:wz:z).
\end{equation}
This map is not defined on the locus $u=v=t=0$; on the complement of
this locus it is not an isomorphism along $z=0$.

We construct $Y'$ as a composition of two blow ups:
\begin{enumerate}
\item we consider the blow up $Y_1\to Y$ of the locus $u=v=t=z=0$,
\item we consider the blow up $Y'\to Y_1$ of the locus $u_1=v_1=t_1=0$ in the chart corresponding to the exceptional divisor defined by $z=0$ (see below).
\end{enumerate}

\subsection{First blow up}
We give equations of the blowup in each chart. Note that around the exceptional divisor we always have $w=1$, and $x=1$ or $y=1$.

\begin{enumerate}
\item In the chart $v=uv_1, t=ut_1, z=uz_1$, the exceptional divisor
is $u=0$, and the blowup is given by
\begin{equation*}
(1-v_1)(t_1-v_1)x^2+(1-v_1)t_1y^2+t_1w^2z_1^2+(t_1-v_1)F(1, t_1, v_1)u^2z_1^2=0.
\end{equation*}
We extend the map \eqref{1map} to the map $Y_1\dashrightarrow V$ defined by
\begin{equation}\label{2map}
(u,v_1, t_1, w\; ;\, x:y, z_1 )\mapsto (1:v_1:t_1, x:y:wz_1:uz_1),
\end{equation}
which is everywhere defined on this chart. Here, we mean that around
the exceptional divisor, we have $w=1$ and $x=1$ or $y=1$.  The fibers of this
map are either points or an $\A^1$ (if $uz_1=0$).  The fiber of
the map $Y_1\to Y$ over the point $(0:0:0:1, x_0:y_0:0)$ is given by
\begin{equation}\label{cubicfiber}
(1-v_1)(t_1-v_1)x_0^2+(1-v_1)t_1y_0^2+t_1z_1^2=0.
\end{equation}
This defines a (singular) cubic surface. By Lemma~\ref{cubicCHtr}
below, this cubic surface is universally $\CH_0$-trivial. We deduce
that the map $Y_1\to Y$ is universally $\CH_0$-trivial over this
chart.

\item In the chart $u=tu_1, v=tv_1, z=tz_1$, the exceptional divisor
is $t=0$, and the blowup is given by
\begin{equation*}
u_1(u_1-v_1)(1-v_1)x^2+u_1(u_1-v_1)y^2+w^2z_1^2+(1-v_1)F(u_1, 1, v_1)t^2z_1^2=0.
\end{equation*}
Similar as in the previous case, we define the map $Y_1\dashrightarrow V$ by
\begin{equation}\label{3map}
(u_1,v_1, t, w\; ;\, x:y, z_1 )\mapsto (u_1:v_1:1, x:y:wz_1:tz_1),
\end{equation}
which is everywhere defined on this chart.  The fibers of this map are
either points or an $\A^1$. On the intersection of the charts, the
maps \eqref{2map} and \eqref{3map} coincide.  The universal
$\CH_0$-triviality of the fibers of the map $Y_1\to Y$ follows from
Lemma~\ref{cubicCHtr}.
\item The chart with exceptional divisor $v=0$ is similar.
\item In the chart $u=zu_1, v=zv_1, t=zt_1$, the exceptional divisor
is $z=0$, and the blowup is given by
\begin{equation}\label{for2blp}
u_1(u_1-v_1)(t_1-v_1)x^2+u_1(u_1-v_1)t_1y^2+t_1w^2+(t_1-v_1)F(u_1, t_1, v_1)z^2=0.
\end{equation}
We define the map $Y_1\dashrightarrow V$ by
\begin{equation}\label{4map}
(u_1,v_1, t_1, w\; ;\, x:y, z )\mapsto (u_1:v_1:t_1, x:y:w:z),
\end{equation}
which is defined everywhere on this chart, except at the locus
$u_1=v_1=t_1=0$.  On the domain of definition, the fibers are points
or lines. The universal $\CH_0$-triviality of the fibers of the map
$Y_1\to Y$ follows from Lemma~\ref{cubicCHtr}.
\end{enumerate}

\begin{lemm}\label{cubicCHtr}
Let $k$ be a field, $a,b\in k$, and $S\subset \P^3_k$ be the cubic
surface defined by 
\begin{equation}\label{lastcubic}
au(u-v)(t-v)+bu(u-v)t+tz^2=0.
\end{equation}
If $a\neq 0$, then $S$ is rational.
\begin{itemize}
\item [\textit{(i)}] If $a\neq 0$ and $(a+b)\neq 0$, then $S$ has three  isolated singular points. The blowup $\tilde S\to S$ at the singular points is smooth and the exceptional divisors are smooth and rational.
\item [\textit{(ii)}] If $a=1, b=-1$, the cubic surface  \eqref{lastcubic} has a unique singular point and it has a universally $\CH_0$-trivial resolution $\tilde S\to S$, given by successive blow ups over this point.
\item [\textit{(iii)}]
If $a=0$, $b=1$, the cubic surface \eqref{lastcubic}
is a union of a plane $t=0$ and a rational quadric surface $u(u-v)+z^2=0$.
\end{itemize} 
\end{lemm}

Note that in the cases \textit{(i)} and \textit{(ii)}, since $\tilde
S$ is smooth and rational, it is universally $\CH_0$-trivial. The
lemma implies that the map $\tilde S\to S$ is universally
$\CH_0$-trivial, hence $S$ is universally $\CH_0$-trivial as well. In
part \textit{(iii)}, we have that $S$ is universally $\CH_0$-trivial
as well from its description.

\begin{proof}
The rational parameterization of $S$ is given by  projection from the point $z=u=v=0$.

Part \textit{(iii)} is straightforward. For part \textit{(i)}, by
direct computation we obtain the following description of the singular
locus:
\begin{itemize}
\item $z=u=v=0$;
\item $z=0, u=0, av-(a+b)t=0$;
\item $z=0, u-v=0, av-(a+b)t=0$.
\end{itemize}
We consider the exceptional divisor of the blowup of the first
point. The other cases are similar, up to a linear change of
variables. Put $c=a+b$. We write the equation of $S$, in the open
chart $t=1$, as
\begin{equation*}
u(u-v)(c-av)+z^2=0.
\end{equation*}
We then have the following charts for the blowup of $u=v=z=0$:
\begin{itemize}
\item In the chart $u=zu_1, v=zv_1$, the blowup
$u_1(u_1-v_1)(c-azv_1)+1=0$ is smooth, and the exceptional divisor
$u_1(u_1-v_1)c+1=0$ is smooth and rational.
\item In the chart $z=uz_1, v=uv_1$, the blowup
$(1-v_1)(c-auv_1)+z_1^2=0$ is smooth, and the exceptional divisor
$(1-v_1)c+z_1^2=0$ is smooth and rational.
\item In the chart $z=vz_1, u=vu_1$, the blowup
$u_1(u_1-1)(c-av)+z_1^2=0$ is smooth, and the exceptional divisor
$(1-u_1)c+z_1^2=0$ is smooth and rational.
\end{itemize}

In the part \textit{(ii)}, the unique singular point of $S$ is given
by $z=u=v=0$. We have the same equations as above, with $c=0, a=1$. We
have additional double point singularities, in the second and the
third chart (for example, the point $z_1=u=v_1=0$ in the second chart
$z_1^2-uv_1(1-v_1)=0$) which are resolved by one single blowup, and
then the exceptional divisor is smooth and rational.
\end{proof}

\subsection{Second blowup}
We consider the chart \eqref{for2blp} and we blow up the locus $u_1=v_1=t_1=0$:
\begin{itemize}
\item In the chart $v_1=u_1v_2, t_1=u_1t_2$, the blowup is given by
\begin{equation*}
u_1^2(1-v_2)(t_2-v_2)x^2+u_1^2(1-v_2)t_2y^2+t_2w^2+(t_2-v_2)F(1, v_2, t_2)u_1^2z_1^2=0.
\end{equation*}
We consider the map $Y'\to V$ defined by
\begin{equation}\label{5map}
(u_1,v_2, t_2, w\; ;\, x:y, z_1 )\mapsto (1:v_2:t_2, u_1x:u_1y:w:u_1z_1),
\end{equation}
which is everywhere defined ($w\neq 0$). The fibers of the map are
points or affine spaces. The fibers of the blow up $Y'\to Y_1$ on this
chart are also either points or affine spaces.
\item The analysis of the other two charts is similar.
\end{itemize}

\section{Appendix: Analysis of singularities}

The main goal of this section is to prove the following result.

\begin{theo}\label{resolution}
Let $X_n$ be as defined in \eqref{eproj}. Then $X_n$ admits a
universally $\CH_0$-trivial resolution $\tilde X_n \to X_n$.
\end{theo}

The following lemma is known:

\begin{lemm}\label{formal}
Let $X$ be a proper variety over a field $k$ of characteristic
zero. If $X$ admits a universally $\CH_0$-trivial resolution $\tilde
X\to X$, then any resolution $X'\to X$ is universally $\CH_0$-trivial.
\end{lemm}
\begin{proof}
We note that by resolution of singularities, there is a smooth
projective variety $\tilde X'$ together with birational morphisms
$\tilde X'\to X'$ and $\tilde X'\to \tilde X$. Then, it is enough to
observe that any birational morphism of smooth projective varieties
over $k$ is universally $\CH_0$-trivial.  For this, we use weak
factorization, the fact that a blowup with smooth center is a
universally $\CH_0$-trivial morphism, and the fact that, by definition,
a composition of two universally $\CH_0$-trivial maps is universally
$\CH_0$-trivial.
\end{proof}

By \cite[Prop.~1.8]{CTP} we have the following criterion: a proper
morphism $\tilde X\rightarrow X$ is universally $\CH_0$-trivial if for
any scheme-theoretic point $P \in X$, the fiber $\tilde{X}_P$ is a
universally $\CH_0$-trivial variety over the residue field $\kappa(P)$.
Using this criterion, in order to prove that a sequence of blowups of
a variety $X$ provides a universally $\CH_0$-trivial resolution
$f:\tilde X\to X$, it is enough to work formally locally on $X$.
Indeed, if $\hat{\OO}_{X,P}$ is the completion of the local
ring ${\OO}_{X,P}$, the fiber of the induced map $\tilde
X\times_X \mathrm{Spec}\,\hat{\OO}_{X,P}\to
\mathrm{Spec}\,\hat{\OO}_{X,P}$ at the closed point of
$\mathrm{Spec}\,\hat{\OO}_{X,P}$ is $\tilde X_P$.

We now analyze different types of singularities that could appear for $X_n$.

\subsection{Equations defining singular locus}
By symmetry, we may work over an open chart $z\neq 0$ of $\P^2$. Then
we have the following equations defining the singular locus:
\begin{multline}\label{eqsingular}
as=bt=cu=dFv=0,\\
\frac{\partial a}{\partial x}s^2+\frac{\partial b}{\partial x}t^2+\frac{\partial c}{\partial x}u^2+\frac{\partial dF}{\partial x}v^2=0,\;
\frac{\partial a}{\partial y}s^2+\frac{\partial b}{\partial y}t^2+\frac{\partial c}{\partial y}u^2+\frac{\partial dF}{\partial y}v^2=0.
\end{multline}

Note that if $P=((x,y),[s:t:u:v])\in X_n$ satisfies that $(x,y)\notin    D:=L_1\cup\dotsm\cup L_m\cup C$, then $P$ is a smooth point. Indeed, we then have $abcdF(P)\neq 0$ and the conditions above 
imply that 
$$
s=t=u=v=0,
$$ 
which is not possible since $s,t,u,v$ are projective coordinates.

Also, if $P\in C$, but not on any line $L_1, \dotsc, L_m$, the conditions \eqref{eqsingular} imply that 
$$
s=t=u=0, \quad v=1, \quad F=0, \quad \frac{\partial dF}{\partial x}(P)=\frac{\partial dF}{\partial y}(P)=0,
$$ 
which is impossible since the conic $C$ is smooth.

Hence we need to analyze the following four types of singularities:
\begin{itemize}
\item over  the generic point of lines $L_i$;
\item over the intersection points  $L_i\cap L_j$;
\item over the tangency point of $C$ and $L_i$;
\item over closed points of lines $L_i$, that are not on other lines or on the conic $C$.
\end{itemize}

Note that by \cite{HPT, HPT-double}, a universally $\CH_0$-trivial resolution exists in the following cases:
\begin{align}\label{eqinHPT}
\begin{split}
&a=yz, b=xy, c=xz, d=1;\\
&a=1,b=xy, c=xz, d=yz.
\end{split}
\end{align}

In the arguments below, up to a linear change of variables $x$ and $y$, we may assume that $\ell_i=x$ and $\ell_j=y$. 

The analysis below provide the following global description of singularities:
\begin{enumerate}
\item the curves $C_i$  \eqref{sconic}, \eqref{sconic2} over the lines $L_i$, some of these curves are singular at a point $P_i$  \eqref{singularpointplus} over the tangency point of $C$ and $L_i$;
\item  singular lines \eqref{singularlines} and \eqref{singularlines2},  over the intersection points of $L_i$ and $L_j$;
\item the curves $D_i$ \eqref{conictang} over the tangency points of $C$ et $L_i$.
\end{enumerate}

We consider the map 
\begin{equation}\label{eresolution}
 X_n'\to X_n
\end{equation} 
given by  successive blow ups of the singular locus in the following order:
\begin{itemize}
\item we blow up lines \eqref{singularlines};
\item we blow up the points $P_i$ and then we blow up the exceptional divisors over $P_i$,
\item we blow up of the proper transform of $C_1$, the proper transform of $C_2, \dotsc$ and then proper transform of $C_m$, 
\item  we blow up 
successively the proper transforms of lines \eqref{singularlines2},
\item  finally, we blow up the proper transforms of the curves $D_i$.
\end{itemize}

We claim that the only singularities of the variety $X_n'$   are  over some intersection points of $L_i$ and $L_j$, the blowup $\tilde X_n$  of these singularities is smooth, and the resulting map $\tilde X_n\to X_n'\to X_n$ is a universally $\CH_0$-trivial resolution.

\subsection{Singularities over   lines $L_i$}
We have two cases to consider:
\begin{enumerate}
\item $\ell_i=x$ divides precisely two among the coefficients $a,b,c$, by symmetry we may assume that $x\,|a,b$, we then write $a=xa_1, b=xb_1$;
\item $x$ divides $d$, by symmetry, we may assume that $x\,|a,d$ and we write $d=xd_1, a=xa_1$.
\end{enumerate}

 Then the analysis of the singular locus in each case is as follows:
\begin{enumerate}
\item The equations \eqref{eqsingular} imply $u=v=0, a_1s^2+b_1t^2=0$. Let $\lambda$ be the common factor of $a_1$ and $b_1$: $\lambda$ is the product of (some of) lines $\ell_1, \dotsc, \ell_m$ and $a_1=\lambda a_2, b_1=\lambda b_2$. We obtain a curve $C_i\subset X_n^{sing}$ in the singular locus $X_n^{sing}$ of $X_n$:
\begin{equation}\label{sconic}
x=0, u=v=0, a_2s^2+b_2 t^2=0.
\end{equation}

 We claim that the blowup  of the curve $C_i$ is smooth at any point of the exceptional divisor that is not over a point of $C$ or a point on another line, and that the corresponding fibers  are universally $\CH_0$-trivial.

For the fibers  over a point $P\in X_n$ lying over a closed point $Q\in L_i$, we work  over the  local ring $\hat{\mathcal O}_{X_n, P}$. Since the residue field of $Q$ is the field of complex numbers, any element of $\hat{\mathcal O}_{\P^2, Q}$, which does not vanish at $Q$, is a square in $\hat{\mathcal O}_{\P^2, Q}$, and hence in $\hat{\mathcal O}_{X_n, P}$. We then obtain the following formal equation:
\begin{equation}\label{analyticsing1}
x s^2+x t^2+u^2+v^2=0,
\end{equation}
 the singularity is defined by $u=v=0, x=0, s^2+t^2=0$. This type of singularity has been already treated in \cite{HPT} (compare with equations \eqref{eqinHPT}), it is resolved with one blow up, and the corresponding fiber is universally $\CH_0$-trivial.

For the fiber over the generic point of $C_i$, it is enough to consider the following charts of the blowup:
\begin{enumerate}
\item 
$a_2s^2+b_2t^2=xw, u=x u_1, v=xv_1.$

The blowup is given by the conditions
\begin{equation*}
\lambda w+cu_1^2+dFv_1^2=0, a_2s^2+b_2t^2=xw
\end{equation*}
and the exceptional divisor is defined by
\begin{equation*}
x=0, \lambda w+cu_1^2+Fdv_1^2=0,a_2s^2+b_2t^2=0
\end{equation*}
This variety is smooth and rational over the generic point of $C_i$ since $\frac{b_2}{a_2}$ is a square at the generic point of $C_i$.
\item $a_2s^2+b_2t^2=uw, x=ux_1, v=uv_1.$ 
 
The blowup is given by the conditions
\begin{equation*}
\lambda x_1w+c+dFv_1^2=0,  a_2s^2+b_2t^2=uw,
\end{equation*}
it is smooth since $a_2, b_2, \lambda$ and $c$ do not vanish at the generic point of $L_i$ and $s\neq 0$ or $t\neq 0$.  The exceptional divisor is also smooth and rational, defined by $u=0$.

The chart with the exceptional divisor defined by $v=0$ is similar.
\item $x=(a_2s^2+b_2t^2)x_1,  u=(a_2s^2+b_2t^2)u_1, v=(a_2s^2+b_2t^2)v_1.$ 

The blowup is given by the condition
\begin{equation*}
\lambda x_1+cu_1^2+dFv_1^2=0,
\end{equation*}
it is smooth. The exceptional divisor is  smooth and rational, defined by
\begin{equation*}
\lambda x_1+cu^2+dFv^2=0, a_2s^2+b_2t^2=0.
\end{equation*}

\end{enumerate}
\item Similarly as in the previous case, let $\lambda$ be the common factor of $a_1$ and $d_1$ and $a_1=\lambda a_2, d_1=\lambda d_2$. We obtain the equation of the singular locus:
\begin{equation}\label{sconic2}
x=0, t=u=0, a_2s^2+d_2Fv^2=0.
\end{equation}
The formal equation at closed points are the same as the equations in the previous case \eqref{analyticsing1}. For the generic fiber, we consider the following chart of the blowup \begin{equation*}
a_2s^2+d_2Fv^2=xw, t=x t_1, u=xu_1.
\end{equation*}
The blowup is given by the conditions
\begin{equation*}
\lambda w+bt_1^2+cu_1^2=0, a_2s^2+d_2Fv^2=xw
\end{equation*}
and the exceptional divisor is defined by
\begin{equation*}
x=0, \lambda w+bt_1^2+cu_1^2=0, a_2s^2+d_2Fv^2=0
\end{equation*}
This variety is smooth and rational over the generic point of $C_i$ since $\frac{d_2F}{a_2}$ is a square at the generic point of $C_i$.

The analysis of the other charts is similar to the previous case.
\end{enumerate}

\subsection{Singularities over intersection points  $L_i\cap L_j$}
Let $Q$ be the intersection point of $L_i$ and $L_j$ and let $P$ be a singular point of $X_n$ over $Q$. 
We have the following cases to consider:
\begin{enumerate}
\item Only two coefficients among $a,b,c,d$ vanish at $Q$.
We then have the following possibilities (up to a symmetry):
$xy|a,b$  or $xy|a,d$ and the corresponding singular lines are given by the (global) conditions
\begin{equation}\label{singularlines}
M_{ij}: x=y=u=v=0 \mbox{ or }x=y=t=u=0.
\end{equation}

Again, working formally locally, we may assume that any function that does not vanish at $Q$ is a square. Hence,  in all cases, up to a symmetry,  we have the following type of local equation:
\begin{equation}\label{analyticsing2}
xy s^2+xy t^2+u^2+v^2=0
\end{equation}
and the singularity is given by $x=y=u=v=0$.  Also we can change variables $s^2+t^2=s_1t_1$ and consider the chart $t_1=1$, so that we have the following local equation:
\begin{equation*}
xys_1+u^2+v^2=0.
\end{equation*}
The map $X_n'\to X_n$, restricted to $\hat{\mathcal O}_{X_n, P}$ is the following composition:
\begin{itemize}
\item  blow up of the line $x=y=u=v=0$;
\item  blow up of the proper transform of $x=s_1=u=v=0$;
\item  blow up of the proper transform of $y=s_1=u=v=0$.
\end{itemize}

By symmetry between $u$ and $v$ we consider the following charts of the first blow up:
\begin{enumerate}
\item $x=yx_1, u=yu_1, v=yv_1$, the equation of the blowup is
$x_1s_1+u_1^2+v_1^2=0$ and the exceptional divisor is given by
$y=0$. This blow up map is universally $\CH_0$-trivial over this
chart. Next we blow up the locus $x_1=s_1=u_1=v_1=0$, corresponding to
the product of a line $y$ and the ordinary double point singularity,
hence the second blowup is smooth, the fibers are universally
$\CH_0$-trivial. By smoothness, the third blow up is universally
$\CH_0$-trivial over this chart.
\item $y=xy_1, u=xu_1, v=xv_1$, the equation of the blowup is $y_1s_1+u_1^2+v_1^2=0$ and the exceptional divisor is given by $x=0$.  Next we blow up the locus $x=s_1=u_1=v_1=0$. We consider the following charts (again, using symmetry between $u_1$ and $v_1$):
\begin{enumerate}
\item $s_1=xs_2, u_1=xu_2, v_2=xv_2$, the equation of the blowup is  $y_1s_2+x(u_2^2+v_2^2)=0$, the exceptional divisor is $x=0$. The fibers of the second blow up are universally $\CH_0$-trivial. Next we blow up the locus $y_1=s_2=u_2=v_2=0$. The charts corresponding to the exceptional divisors $y_1=0$ and $s_2=0$ are smooth, and the fibers are universally $\CH_0$-trivial:
\begin{itemize}
\item  $s_2=y_1s_3, u_2=y_1u_3, v_2=y_1v_3$, the blowup is given by $s_3+x(u_3^2+v_3^2)=0$; 
\item  $y_1=s_2y_3, u_2=s_2u_3, v_2=s_2v_3$, the blowup is given by $y_3+x(u_3^2+v_3^2)=0$; 
\end{itemize}
In the chart $y_1=u_2y_3, s_2=u_2y_3, v_2=u_2v_3$  we have the following equation: $y_3s_3+x(1+v_3^2)=0$, the ordinary double singularities $v_3=\pm i, x=y_3=s_3=0$ are resolved after one blowup, and the exceptional divisor is rational.
\item $x=s_1x_2, u_1=s_1u_2, v_2=s_1v_2$, the equation of the blowup is  $y_1+s_1(u_2^2+v_2^2)=0, x=s_1x_2$, the exceptional divisor is $s_1=0$, the fibers are universally $\CH_0$-trivial. The blowup is smooth, hence the third blow up is universally $\CH_0$-trivial over this chart.
\item $x=u_1x_2, s=u_1s_2, v_2=u_1v_2$, the equation of the blowup is $y_1s_2+u_1+u_1v_2^2=0, x=u_1x_2$,  the exceptional divisor is given by $u_1=0$. Next we blow up $y_1=s_2=u_1=v_2=0$. Note that this chart is smooth along this locus, hence the third blow up is universally $\CH_0$-trivial in this chart. The remaining singularity $y_1=s_2=u_1=0, v_2=\pm i$ is resolved as at the end of the case (i) above.

\end{enumerate}
\item $x=ux_1, y=uy_1, v=uv_1$, the equation of the blow up is
$x_1y_1s_1+1+v_1^2=0$, it is smooth and the exceptional divisor is
given by $u=0$. This blow up map is universally $\CH_0$-trivial in
this chart. Since the first blow up is smooth over this chart, the second and third blow ups are universally $\CH_0$-trivial.
\end{enumerate}

\item Only three coefficients among $a,b,c,d$ vanish at $Q$.
Then, we may assume that $xy\,|a, x\,|b, y\,|c$. The case when $x$ or $y$ divides $d$ is similar.
The singular locus over $x=y=0$ is given by $t=u=v=0$, hence, it is the point of  the intersection of the curves $C_i$ and $C_j$. The restriction of the map $X_n'\to X_n$  to $\hat{\mathcal O}_{X_n, P}$ is the composition of the blow up of $C_i$ and the proper transform of $C_j$.
Note that over the point $P$ we have the following formal equation:
\begin{equation}\label{analyticsing3}
xy s^2+x t^2+y u^2+v^2=0
\end{equation}
The same type of formal equation correspond to the case  considered in \cite{HPT} (compare with equations \eqref{eqinHPT}), where it is showed that the two blow ups as above provide  a universally $\CH_0$-trivial resolution.

\item All coefficients $a,b,c,d$ vanish at $Q$. We then assume $a=xa_1, b=xb_1, c=yc_1, d=yd_1$ and we have the following equation for $X_n$
\begin{equation*}
xa_1s^2+xb_1t^2+yc_1u^2+yd_1Fv^2=0
\end{equation*}
and the expression of the singular locus
\begin{equation}\label{singularlines2}
x=y=0, a_1s^2+b_1t^2=0, c_1u^2+d_1Fv^2=0
\end{equation}
Since $a_1, b_1, c_1, d_1F$ evaluated at $Q$ are nontrivial constants, this singular locus is the union of four lines, and the points with coordinates $s=t=0$ or $u=v=0$ correspond to the intersection points of these lines with the curves $C_i$ and $C_j$.

 We now describe the the formal equation.  Changing  variables, we may replace $a_1s^2+b_1t^2$ by $s_1t_1$ and $c_1u^2+d_1Ft^2$ by $u_1v_1$, and, by symmetry, we may consider the affine chart $t_1=1$. We then have the following formal equation
\begin{equation}\label{analyticsing4}
xs_1+yu_1v_1=0
\end{equation}
The singular locus is the union of two lines $x=y=s_1=u_1=0$ et $x=y=s_1=v_1=0$.
The restriction of the map $\tilde X_n\to X_n$ is as follows:
\begin{enumerate}
\item blow up of the locus $x=s_1=u_1=v_1=0$ (we separate two lines);
\item blow up of the strict transform of the lines $x=y=s_1=u_1=0$ and $x=y=s_1=v_1=0$.
\end{enumerate}
(Note that only one of the blowups of $C_i$ and $C_j$ is not an isomorphism for the case $t_1=1$ we consider). 
By symmetry, we have the following charts of the first blowup to consider: 
\begin{enumerate}
\item $s=xs_2, u=xu_2, v=xv_2$, the equation of the blowup is $s_2+yu_2v_2=0$,  which is smooth, the exceptional divisor is $x=0$, it is smooth and rational; by smoothness, the second blowup is universally $\CH_0$-trivial over this chart; 
\item $s=u_1s_2, x=u_1x_2, v=u_1v_2$, the equation of the blowup is $x_2s_2+yv_2=0$, the exceptional divisor is rational, given by $u_1=0$; the singular locus $x_2=s_2=y=v_2=0$ is resolved by the second blowup, and the exceptional divisor is rational.
\end{enumerate}
From the equations above, the fibers of the map $X_n'\to X_n$ at $P$ are universally $\CH_0$-trivial over this chart.

\end{enumerate}

\subsection{Singularities over the tangency point of $C$ and $L_i$}
Assume $C$ is tangent to $L_i$ at the point $Q$. We have the following cases:
\begin{enumerate}
\item $d$ vanishes at $Q$. By symmetry, we assume $x|a$ and write $a=xa_1, d=xd_1$. Then the conditions  \eqref{eqsingular} imply 
\begin{equation*}
t=u=0,\, a_1s^2+d_1F(Q)v^2=0,
\end{equation*} hence  we obtain a point 
\begin{equation}\label{singularpointplus}
P_i: s=t=u=x=y-1=0
\end{equation} 
 on the curve $C_i$ in  the case \eqref{sconic2}.
The local form  is as follows:
\begin{equation}
xs^2+t^2+u^2+xFv^2=0, \mbox{ singular locus: }s=t=u=x=y-1=0
\end{equation} 
This is the same type of formal equation as for the quadric bundle we considered in \cite{HPT-double} (see equations \eqref{eqinHPT}), by {\it loc. cit.}, the map \eqref{eresolution} provide  a universally $\CH_0$-trivial resolution: we  blow up successively the point $P_i$,  then the exceptional divisor of the blowup and the proper transform of $C_i$.

\item $d$ does not vanish at $Q$. We may then assume $a=xa_1, b=xb_1$, so that the conditions  \eqref{eqsingular} imply 
\begin{equation}\label{conictang}
D_i: u=0, a_1s^2+b_1t^2+d \frac{\partial F}{\partial x}(Q)v^2=0. 
\end{equation}

Arguing as in the previous cases, we obtain the following formal local form of the singular locus:

\begin{align}\label{analyticsing6}
\begin{split}
& x s^2+ x t^2+u^2+ F v^2=0, 
\mbox{ singular locus: }  u=0, s^2+t^2+\frac{\partial F}{\partial x}(Q)v^2=0.
\end{split}
\end{align}
This is the same type of formal equation as for the quadric bundle we considered in \cite{HPT} (see equations \eqref{eqinHPT}), by {\it loc. cit.}, the map \eqref{eresolution} provide  a universally $\CH_0$-trivial resolution.
\end{enumerate} 

This finishes the proof of Theorem \ref{resolution}.

\bibliographystyle{plain}
\bibliography{bidegree23}
\end{document}